\documentclass{amsart}
\usepackage{amsmath}
\usepackage{framed,comment,enumerate}
\usepackage[pdftex]{graphicx}
\usepackage{epstopdf}
\usepackage{xcolor, framed}
\usepackage{color}

\def\R {\mathbb{R}}
\def\N{\mathbb{N}}
\def\eps{\varepsilon}
\def\PB{\partial B_1}

\def\PosS{\{u>0\}}
\def\ConS{\{u=0\}}
\def\PPosS{\partial\{u>0\}}

\def\Diri{|\nabla u|^2}

\def\EGa{\mathcal{E}_\gamma}
\def\KGa{\mathcal{K}(\gamma)}

\def\K{\mathcal{K}}
\def\HS{\mathcal{HS}}
\def\P{\mathcal{P}}
\def\SP{\mathcal{SP}}

\def\hem{\hspace{0.5em}}
\def\vem{\vspace{0.6em}}

\newtheorem{thm}{Theorem}[section]
\newtheorem{prop}[thm]{Proposition}

\newtheorem{cor}[thm]{Corollary}
\newtheorem{lem}[thm]{Lemma}
\theoremstyle{definition}
\newtheorem{defi}[thm]{Definition}
\numberwithin{equation}{section}

\theoremstyle{remark}
\newtheorem{rem}[thm]{Remark}

\title{Concentration of  cones in the Alt-Phillips problem} 

\author{Ovidiu Savin}
\address{Department of Mathematics,	Columbia University, New York, USA}
\email{savin@math.columbia.edu}

\author{Hui Yu}
\address{Department of Mathematics,	National University of Singapore, Singapore}
\email{huiyu@nus.edu.sg}

\begin{document}

\begin{abstract}
We study minimizing cones in the Alt-Phillips problem when the exponent $\gamma$ is close to $1$.

When $\gamma$ converges to $1$, we show that the cones concentrate around \textit{symmetric} solutions to the classical obstacle problem. To be precise, the limiting profiles are radial in a subspace and invariant in directions perpendicular to that subspace.

\end{abstract}
\maketitle

\section{Introduction}

For an exponent $\gamma\in[0,2)$, the \textit{Alt-Phillips problem} studies minimizers of the following energy
\begin{equation}
\label{EqnAP}
\mathcal{E}_\gamma(u;\Omega):=\int_{\Omega}\frac{\Diri}{2}+u^\gamma\chi_{\{u>0\}},
\end{equation} 
where $\Omega$ is a domain in $\R^{d}$, and $u$ is a \textit{non-negative} function. 

Since its introduction by Phillips \cite{P} and Alt-Phillips \cite{AP}, this problem has received intense attention in the past few decades. See, for instance, \cite{AS, B, DS, FeRo, FeY, FeY2,KSP, WY,WaY}. For a minimizer $u$, the focus is on regularity properties of its \textit{free boundary} $\PPosS$, separating the \textit{positive set} $\PosS$ from the \textit{contact set} $\ConS$. Thanks to Weiss \cite{W}, the infinitesimal behavior of the free boundary is given by homogeneous minimizers, known as \textit{minimizing cones}. 

The simplest  cone is the \textit{flat cone}, the one-dimensional minimizer of \eqref{EqnAP} that vanishes on a half space (see \eqref{EqnFlatConeAP}). If a minimizer $u$ is close to the flat cone, then its free boundary $\PPosS$ is smooth \cite{AP, D, DS, ReRo}. For certain ranges of $d$ and $\gamma$, the only minimizing cone is the flat cone \cite{AP, CJK, JS}. 

In general, there are non-flat minimizing cones for \eqref{EqnAP}, known as \textit{singular cones} \cite{DJ, SY2}. They provide singularity models for this problem. Unfortunately, very little is known about these cones. As a result, we face  serious challenges in the study of singular  free boundary points for general $\gamma$. 

\vem

The exception is when $\gamma=1$, corresponding to the \textit{classical obstacle problem}. For this problem, homogeneous solutions are fully classified \cite{C, PSU}. This leads to  a lot of developments on the singular part of the free boundary \cite{C, CSV, FRoSe, FSe, L, M, SY1}.  

To be precise, cones in the classical obstacle problem belong to two families (see Proposition \ref{PropObPCones}). One family consists of flat cones, known as \textit{half-space solutions} in this problem. They are of the form  
\begin{equation}
\label{EqnHalfSpaceSolutions}
\frac12[(x\cdot e)_+]^2\hem \text{ for some unit vector }e.
\end{equation} 
The other family consists of singular cones,  known as \textit{parabola solutions}. These solutions are of the form 
\begin{equation}
\label{EqnParabolaSolutions}
\frac12x\cdot Ax, \text{ where  }A\ge0 \text{ and }\mathrm{trace}(A)=1.
\end{equation}

Note that half-space solutions and parabola solutions each form a connected family. Even modulo rotational symmetry, there are infinitely many parabola solutions. We also remark that the contact set of any parabola solution, being a proper subspace of $\R^d$,  has zero measure. See, for instance, Proposition \ref{PropConvexityOfSolutions}. 

\vem

To understand cones in the Alt-Phillips problem \eqref{EqnAP}, it is natural to first consider the situation when $\gamma$ is close to $1$. 

Recently, we constructed axially symmetric cones to \eqref{EqnAP}, when  $\gamma<1$ is close to $1$ and $d\ge4$ \cite{SY2}. For each $\gamma$ in this range, there is a cone whose contact set has zero measure, similar to parabola solutions in the obstacle problem (see \eqref{EqnParabolaSolutions}). Interestingly, there is also a cone whose contact set has positive measure, similar to the De Silva-Jerison cone \cite{DJ} in the Alt-Caffarelli problem ($\gamma=0$). This is the first free boundary problem that allows both possibilities. 

This shows that singular cones  for \eqref{EqnAP}, even when $\gamma$ is close to $1$, exhibit rich and complicated behavior. The first step in understanding these cones is to study their limiting behavior when $\gamma$ converges to $1$.

\vem

As $\gamma\to 1$, minimizing cones for \eqref{EqnAP} concentrate around cones for the obstacle problem (see Proposition \ref{PropCompactness}). Surprisingly, only cones with strong symmetric properties can show up as limits of cones to \eqref{EqnAP} when $\gamma\to1. $ 

To be precise, for  $k\in\{0,1,\dots,d\}$, we define 
\begin{equation}
\label{Eqn0Cone}
P_0(x):=\frac12[(x_1)_+]^2,
\end{equation} 
and
\begin{equation}
\label{EqnkCone}
P_k(x):=\frac{1}{2k}\sum_{1\le j\le k}x_j^2 \hem\text{ for }k=1,2,\dots,d.
\end{equation} 
Our \textbf{main result} reads
\begin{thm}
\label{ThmMain}
Given $\eps>0$, there is $\eta>0$, depending only on $\eps$ and the dimension $d$, such that if $u$ is a minimizing cone for the Alt-Phillips problem \eqref{EqnAP} with
$$
\gamma\in(1-\eta,1)\cup(1,1+\eta),
$$
then, up to a rotation, we have
$$
\|u-P_k\|_{L^\infty(B_1)}<\eps 
$$
 for some $k\in\{0,1,\dots,d\}.$
\end{thm} 

\begin{rem}
To the knowledge of the authors, this is the first result on concentration-induced symmetry in free boundary problems. 

The cones in \eqref{EqnkCone} are fully radial in a subspace of $\R^d$ and invariant along directions perpendicular to this subspace. 

For each integer $k\in\{0,1,\dots,d\}$, up to a rotation,  there is exactly one cone, $P_k$, that can show up as a limiting profile with $$\mathrm{dim}(\{P_k=0\})=d-k.$$
\end{rem}
\vem
The heart of  Theorem \ref{ThmMain} is the following integral inequality on parabola solutions in \eqref{EqnParabolaSolutions}. This inequality is stable in the natural topology of our problem (see Lemma \ref{LemMagicLemmaForV}), and the equality case captures cones with symmetry from \eqref{EqnkCone}.
\begin{lem}
\label{LemMagicLemma}
Suppose that $p$ is a parabola solution to the classical obstacle problem in $\R^d$ from \eqref{EqnParabolaSolutions}, then
\begin{equation*}
\int_{\partial B_1}\frac{|\nabla p|^2}{p}(p-\frac{1}{2d}|x|^2)\ge0.
\end{equation*} 
The equality  holds if and only if, up to a rotation, 
$$
p=P_k
$$
 for some $k=1,2,\dots,d$ from \eqref{EqnkCone}.
\end{lem} 

\vem
This paper is \textbf{structured as follows}: In Section \ref{SectionPreliminaries}, we collect some preliminaries on the Alt-Phillips problem \eqref{EqnAP} and introduce some notations. Section \ref{SectionInequality} is devoted to the proof of Lemma \ref{LemMagicLemma}. With this, we prove Theorem \ref{ThmMain} in Section \ref{SectionConesWithSymmetry}.

\section{Preliminaries}
\label{SectionPreliminaries}

In Subsection \ref{SubsectionAP}, we collect some useful facts about the Alt-Phillips functional in \eqref{EqnAP}. We then transform the minimizers to a different form, to take advantage of the proximity of $\gamma$ to $1$. Properties of this transformed solution are presented in Subsection \ref{SubsectionTransformedSolution}. In Subsection \ref{SubsectionObP}, we recall properties of the classical obstacle problem.

\subsection{The Alt-Phillips functional}
\label{SubsectionAP}

Recall the Alt-Phillips functional $\mathcal{E}_\gamma$ from \eqref{EqnAP}. In this work, we focus on the case when  $\gamma$ is close to $1$. In particular, we always assume
\begin{equation*}
\gamma\in[1/2,3/2].
\end{equation*} 

Suppose that $u$ is a minimizer of $\EGa$ in $B_1$, then it satisfies the \textit{Euler-Lagrange equation} (see \cite{AP})
\begin{equation}
\label{EqnELEquation}
\Delta u=\gamma u^{\gamma-1}\chi_{\PosS}, \hem u\ge 0\hem \text{ in }B_1.
\end{equation}

For each $\gamma$, we define the \textit{scaling parameter} as
\begin{equation}
\label{EqnBeta}
\beta:=\frac{2}{2-\gamma}=2-\frac{2(1-\gamma)}{2-\gamma}.
\end{equation} 
When $\gamma$ is close to $1$, this parameter $\beta$ is close to $2$. 

The scaling parameter arises from  the symmetry of $\EGa$. To be precise, for each $r>0$, if we take
$$
u_r(x):=r^{-\beta}u(rx),
$$
then we have 
$$
\EGa(u_r;B_1)=r^{2-2\beta-d}\EGa(u;B_r).
$$
In particular, if $u$ is a non-trivial minimizing cone for $\EGa$ in \eqref{EqnAP}, then \textit{$u$ must be $\beta$-homogeneous}. As a consequence, the flat cone for $\EGa$ is of the form
\begin{equation}
\label{EqnFlatConeAP}
u(x)=c_\gamma[(x\cdot e)_+]^\beta \text{ for some unit vector}\hem e,
\end{equation} 
where the coefficient $c_\gamma$ satisfies
$
c_\gamma^{2-\gamma}=(2-\gamma)^2/2.
$

Compactness of  minimizers follows from the following:
\begin{thm}[See \cite{AP}]
\label{ThmRegularityOfMinimizerAP}
Suppose that $u$ is a minimizer of $\EGa$ from \eqref{EqnAP} in $B_1\subset\R^d$. 

If $0\in\PPosS$, then there is a dimensional constant $C_d$ such that 
\begin{enumerate}
\item{If $\gamma\in[1/2,1]$, then $\|u\|_{C^{1,\beta-1}(B_{1/2})}\le C_d$;}
\item{If $\gamma\in(1,3/2]$, then $\|u\|_{C^{2,\beta-2}(B_{1/2})}\le C_d.$}
\end{enumerate}
Moreover, we have
$$
\sup_{B_{1/2}\cap\PosS}\Diri/u^\gamma+\sup_{B_{1/2}\cap\PosS}|D^2u|/u^{\gamma-1}\le C_d.
$$
\end{thm} 
Recall the parameter $\beta$ from \eqref{EqnBeta}.

When $\gamma\ge1,$ the following dichotomy is a consequence of  the convexity of global solutions:
\begin{prop}[\cite{B, C, PSU, WY}]
\label{PropConvexityOfSolutions}
If $u$ is a minimizing cone for  $\EGa$ in \eqref{EqnAP} with $\gamma\in[1,2)$, then either $u$ is the flat cone in \eqref{EqnFlatConeAP}, or $\{u=0\}$ has zero measure.
\end{prop}

\subsection{The transformed solution}
\label{SubsectionTransformedSolution}
Suppose $v$ is a minimizer of $\EGa$ in \eqref{EqnAP}, we define the \textit{transformed solution} as
\begin{equation}
\label{EqnTransformedSolution}
u:=\frac{1}{\gamma(2-\gamma)}v^{\frac{2}{\beta}},
\end{equation} 
where $\beta$ is from \eqref{EqnBeta}.  If $v$ is $\beta$-homogeneous, then $u$ is $2$-homogeneous. When $\gamma$ is close to $1$, or equivalently, when $\beta$ is close to $2$, this transformed solution solves an equation similar to the classical obstacle problem. See equations \eqref{EqnTransformedEquation}  and \eqref{EqnClassicalObstacleProblem}.

The flat cone in \eqref{EqnFlatConeAP} is transformed to 
\begin{equation}
\label{EqnFlatConeTransformed}
u(x)=\frac{2-\gamma}{2\gamma}[(x\cdot e)_+]^2,
\end{equation} 
where $e$ is a unit vector. 

To simplify our exposition, we introduce the class of transformed cones:
\begin{defi}
\label{DefCones}
Suppose that $v$ is a minimizing cone for $\EGa$ from  \eqref{EqnAP} with  $\gamma\in[1/2,3/2]$, and that $u$ is related to $v$ as in \eqref{EqnTransformedSolution}. 

Then we call $u$ a \textit{transformed cone for }$\EGa$, and write
$$
u\in\mathcal{K}(\gamma).
$$
\end{defi}

Following \eqref{EqnELEquation} and Theorem \ref{ThmRegularityOfMinimizerAP}, we have
\begin{prop}
\label{PropTransformed}
Suppose that $u\in\mathcal{K}(\gamma)$ with $\gamma\in[1/2,3/2]$.

Then $u$ is a $2$-homogeneous solution to 
\begin{equation}
\label{EqnTransformedEquation}
\Delta u+\frac{\beta-2}{2}\frac{|\nabla u|^2}{u}=\chi_{\{u>0\}}, \hem u\ge0\hem \text{ in }\R^d.
\end{equation}

Moreover, we have
$$
u\in C^{1,1}_{loc}(\R^d)
$$ 
with 
$$
|\nabla u|^2/u+|D^2u|\le C_d \text{ in }\PosS
$$
for a dimensional constant $C_d$.
\end{prop} 
Recall the scaling parameter $\beta$ from \eqref{EqnBeta}.

Following Proposition \ref{PropConvexityOfSolutions}, we have
\begin{prop}
\label{PropConvexityTransformed}
Suppose that $u\in\KGa$ with $\gamma\in[1,3/2]$, then
either $u$ is the flat cone in \eqref{EqnFlatConeTransformed}, or $\ConS$ has zero measure. 
\end{prop}

\subsection{The obstacle problem}
\label{SubsectionObP}
When $\gamma=1$, the Alt-Phillips functional $\EGa$ in \eqref{EqnAP} becomes the energy for the classical obstacle problem \cite{C, PSU}.

For this exponent, the scaling parameter in \eqref{EqnBeta} satisfies 
$
\beta=2.
$
The equation in \eqref{EqnTransformedEquation} takes the simple form
\begin{equation}
\label{EqnClassicalObstacleProblem}
\Delta u=\chi_{\PosS},\hem u\ge 0\hem \text{ in }\R^d.
\end{equation}

We have seen  half-space solutions and parabola solutions in \eqref{EqnHalfSpaceSolutions} and \eqref{EqnParabolaSolutions}. Recall the cones with symmetry from \eqref{Eqn0Cone} and \eqref{EqnkCone}. For brevity, we introduce the following notations:
\begin{defi}
\label{DefObPCones}
The collection of half-space solutions for the classical obstacle problem is denoted by $\mathcal{HS}$, namely,
$$
\mathcal{HS}:=\{\frac12[(x\cdot e)_+]^2: \hem e\in\R^d, \hem |e|=1\}.
$$
The collection of parabola solutions is denoted by $\mathcal{P}$, namely,
$$
\mathcal{P}:=\{\frac12x\cdot Ax:\hem A\ge0,\hem \mathrm{trace}(A)=1\}.
$$
The collection of parabola solutions with symmetry is denoted by $\mathcal{SP}$, namely,
$$
\mathcal{SP}:=\{p\in\mathcal{P}: \hem \text{up to a rotation,}\hem p=P_k \text{ for some }k\in\{1,2,\dots,d\} \text{ from }\eqref{EqnkCone}\}.
$$
\end{defi}

Recall the space of cones $\mathcal{K}(\gamma)$ from Definition \ref{DefCones}. We have
\begin{prop}[See \cite{C, PSU}]
\label{PropObPCones}
Suppose that $u\in\mathcal{K}(2)$. Then 
$
u\in\mathcal{HS}\cup\mathcal{P}.
$
\end{prop} 

Proposition \ref{PropTransformed} provides enough compactness for the following:
\begin{prop}
\label{PropCompactness}
Given a sequence $u_n\in\mathcal{K}(\gamma_n)$ (see Definition \ref{DefCones}) with $\gamma_n\to1,$ we have, up to a subsequence, 
$$
u_n\to u_\infty \text{ locally uniformly in }\R^d
$$
for some $u_\infty\in\mathcal{K}(1)$,
\end{prop}
\section{An integral inequality}
\label{SectionInequality}

This section is devoted to the proof of  Lemma \ref{LemMagicLemma}.

The following allows us to perform an induction on the dimension $d$:
\begin{lem}
\label{LemIntegralInDifferentDimensions}
For $d\ge2,$ suppose that $p\in\P$ is invariant in the $e_d$-direction, that is, 
$$
p(x_1,x_2,\dots,x_d)=q(x_1,x_2,\dots,x_{d-1}),
$$
for some $q\in\P$ in $\R^{d-1}$.

Then 
$$
\int_{\partial B_1}\frac{|\nabla p|^2}{p}(p-\frac{1}{2d}|x|^2)d\mathcal{H}^{d-1}=\alpha_d\int_{\partial B_1\cap\{x_d=0\}}\frac{|\nabla q|^2}{q}\cdot(q-\frac{1}{2(d-1)}|(x_1,x_2,\dots x_{d-1})|^2)d\mathcal{H}^{d-2}
$$
for a positive dimensional constant $\alpha_d$.
\end{lem} 
Recall the space of parabola solutions $\P$ from Definition \ref{DefObPCones}.
\begin{proof}
Introduce a variable $\theta\in[-\pi/2,\pi/2]$ such that 
$$
x_d=\sin(\theta) \text{ on }\partial B_1,
$$
then we have
\begin{equation}
\label{EqnIntegralInDifferentDimensions}
\int_{\partial B_1}\frac{|\nabla p|^2}{p}(p-\frac{1}{2d}|x|^2)=\int_{-\pi/2}^{\pi/2}\int_{\partial B_1\cap\{\theta=s\}}\frac{|\nabla p|^2}{p}(p-\frac{1}{2d})dH^{d-2}ds
\end{equation}

Along $\partial B_1\cap\{\theta=s\}$, since $p$ is independent of $x_d$ and is $2$-homogeneous, we have
$$
\frac{|\nabla p|^2}{p}|_{\theta=s}=\frac{|\nabla q|^2}{q}|_{\theta=0},\hem \text{ and }\hem
p|_{\theta=s}=q|_{\theta=0}\cdot\cos^2(s).
$$
As a result, we have
$$
\int_{\partial B_1\cap\{\theta=s\}}\frac{|\nabla p|^2}{p}(p-\frac{1}{2d})dH^{d-2}=\int_{\partial B_1\cap\{\theta=0\}}\frac{|\nabla q|^2}{q}\cdot(q\cos^2(s)-\frac{1}{2d})\cos^{d-2}(s)dH^{d-2}.
$$

Putting this into \eqref{EqnIntegralInDifferentDimensions}, we get
\begin{align*}
\int_{\partial B_1}\frac{|\nabla p|^2}{p}(p-\frac{1}{2d}|x|^2)&=\int_{-\pi/2}^{\pi/2}\cos^{d}(s)ds\cdot\int_{\partial B_1\cap\{\theta=0\}}\frac{|\nabla q|^2}{q}q\\
&-\int_{-\pi/2}^{\pi/2}\cos^{d-2}(s)ds\cdot\int_{\partial B_1\cap\{\theta=0\}}\frac{|\nabla q|^2}{q}\frac{1}{2d}.
\end{align*}
With the recursive formula $\int_{-\pi/2}^{\pi/2}\cos^{d}(s)ds=\frac{d-1}{d}\int_{-\pi/2}^{\pi/2}\cos^{d-2}(s)ds$, we conclude
$$
\int_{\partial B_1}\frac{|\nabla p|^2}{p}(p-\frac{1}{2d}|x|^2)=\frac{d-1}{d}\int_{-\pi/2}^{\pi/2}\cos^{d-2}(s)ds\cdot\int_{\partial B_1\cap\{x_d=0\}}\frac{|\nabla q|^2}{q}\cdot(1-\frac{1}{2(d-1)}).
$$
\end{proof} 

Recall the radial solution $P_d$ from \eqref{EqnkCone}
\begin{equation}\label{EqnPD}
P_d(x)=\frac{1}{2d}|x|^2.
\end{equation}
To connect  $p\in\P$ (see Definition \ref{DefObPCones}) to a solution that is degenerate along an axis, define
\begin{equation}
\label{EqnInterpolation}
p_t(x):=tp(x)+(1-t)P_d(x)\hem \text{ for }t\ge0.
\end{equation} 
As long as $p_t\ge0$, we have 
$
p_t\in\P.
$
In particular, this holds for all $t\in[0,1].$

Correspondingly, we define
\begin{equation}
\label{EqnQInterpolation}
Q(t):=\int_{\partial B_1}\frac{|\nabla p_t|^2}{p_t}(p_t-\frac{1}{2d}|x|^2) \hem\text{ as long as }p_t\ge0. 
\end{equation} 
Following Proposition \ref{PropTransformed}, the integrand is bounded by a dimensional constant.

To prove Lemma \ref{LemMagicLemma} is to show that $Q(1)\ge0$, and to characterize the equality case.

We give an alternative expression of this quantity $Q$:
\begin{lem}
\label{LemMeaningfulQ}
For $p\in\P$, let $p_t$ and $Q$ be as in \eqref{EqnInterpolation} and \eqref{EqnQInterpolation}. Then we have
$$
Q(t)=t^2\cdot[\int_{\partial B_1}|\nabla_\tau p|^2+4\int_{\partial B_1}(p-P_d)^2-\frac{1}{2d}\int_{\partial B_1}\frac{|\nabla_\tau p|^2}{p_t}]
$$
as long as $p_t\ge0.$ 

Here $\nabla_\tau$ denotes the gradient along directions tangential to $\PB$.
\end{lem} 
Recall the class $\P$ from Definition \ref{DefObPCones}. Recall the radial solution $P_d$ from \eqref{EqnPD}.
\begin{proof}
With \eqref{EqnQInterpolation}, we have
\begin{equation}
\label{EqnFirstExpansionOfQ}
Q(t)=\int_{\PB}|\nabla p_t|^2-\frac{1}{2d}\int_{\PB}\frac{|\nabla p_t|^2}{p_t}.
\end{equation}

Note that $|\nabla p_t|^2=|\frac{\partial}{\partial r}p_t|^2+|\nabla_\tau p_t|^2$. 
With the homogeneity of $p_t$, we see that $\frac{\partial}{\partial r}p_t=2p_t$ on $\PB$. With \eqref{EqnInterpolation}, we see that $|\nabla_\tau p_t|=t|\nabla_\tau p|$. As a result, we have
\begin{align*}
|\nabla p_t|^2&=4p_t^2+t^2|\nabla_\tau p|^2\\
&=4p_tP_d+4tp_t(p-P_d)+t^2|\nabla_\tau p|^2\\
&=\frac{4}{2d}p_t+4tp_t(p-P_d)+t^2|\nabla_\tau p|^2 \hem\text{ on }\PB,
\end{align*}
and
$$
\frac{|\nabla p_t|^2}{p_t}=4p_t+t^2\frac{|\nabla_\tau p|^2}{p_t} \hem\text{ on }\PB.
$$
Putting these into \eqref{EqnFirstExpansionOfQ}, we get
\begin{equation}
\label{EqnSecondExpansionOfQ}
Q(t)=t^2\cdot[\int_{\PB}|\nabla_\tau p|^2-\frac{1}{2d}\int_{\PB}\frac{|\nabla_\tau p|^2}{p_t}]+4t\int_{\PB}p_t(p-P_d).
\end{equation}

\vem

From \eqref{EqnInterpolation}, we see that $p_t=tp+\frac{1}{2d}(1-t)$ on $\PB$. Since $(p-P_d)$ is a $2$-homogeneous harmonic function, we have
$
\int_{\PB}(p-P_d)=0.
$
Thus the last integral in \eqref{EqnSecondExpansionOfQ} satisfies
$$
\int_{\PB}p_t(p-P_d)=t\int_{\PB}p(p-P_d)=t\int_{\PB}(p-P_d)^2.
$$
Combining this with \eqref{EqnSecondExpansionOfQ}, we get the desired expansion of $Q$. 
\end{proof}

We now give the proof of Lemma \ref{LemMagicLemma}:
\begin{proof}[Proof of Lemma \ref{LemMagicLemma}]
We argue by an induction on the dimension $d$. 

In $\R^1$, there is only one parabola solution, namely, $x^2/2$. The conclusion follows. 

Suppose that we have established Lemma \ref{LemMagicLemma} in dimensions $1,2,\dots, d-1$, below we prove the result in $\R^d$. 

\vem

For $p_t$ from \eqref{EqnInterpolation}, define 
\begin{equation}
\label{EqnBT}
\bar{t}:=\sup\{t:\hem p_t\ge0 \text{ on }\PB\}.
\end{equation}  
Since  $p\in\P$ (see Definition \ref{DefObPCones}), we have $$\bar{t}\ge1.$$ 

Moreover, by definition of $\bar{t}$, we have
$$
p_{\bar t}\in\P \hem \text{ with }\hem p_{\bar t}(e)=0 
$$
at some $e\in\PB$. This implies that $p_{\bar t}$ is invariant along the $e$-direction. Our induction hypothesis and   Lemma \ref{LemIntegralInDifferentDimensions} give, for the quantity $Q$ in \eqref{EqnQInterpolation},
\begin{equation}
\label{EqnInductionHypothesis}
Q(\bar t)\ge 0, \hem\text{ and equality holds only when }p_{\bar t}\in\SP.
\end{equation} 
Recall the space of symmetric cones $\SP$ from Definition \ref{DefObPCones}. 

\vem
Based on the value of $\bar t$ in \eqref{EqnBT}, we consider two cases.

If $\bar{t}=1$, then $p=p_{\bar t}$. The desired conclusion follows by \eqref{EqnInductionHypothesis}.  

It remains to consider the case when $\bar{t}>1$.

In this case, define
\begin{equation}
\label{Eqnq}
q(t):=Q(t)/t^2=\int_{\partial B_1}|\nabla_\tau p|^2+4\int_{\partial B_1}(p-P_d)^2-\frac{1}{2d}\int_{\partial B_1}\frac{|\nabla_\tau p|^2}{p_t},
\end{equation}
where we used Lemma \ref{LemMeaningfulQ}. 

With \eqref{EqnInterpolation}, we have $p_0=\frac{1}{2d}$ on $\PB$. This implies
\begin{equation}
\label{Eqnq0}
q(0)=4\int_{\partial B_1}(p-P_d)^2\ge0.
\end{equation}
At $\bar t>1$, we have
\begin{equation}
\label{Eqnqbt}
q(\bar t)=Q(\bar t)/\bar{t}^2\ge0
\end{equation}
by \eqref{EqnInductionHypothesis}.

Moreover, direct computation with \eqref{EqnInterpolation} and \eqref{Eqnq} gives, for $t\in(0,\bar t)$,
$$
q''(t)=-\frac{1}{d}\int_{\PB}\frac{|\nabla_\tau p|^2}{p_t^3}(p-P_d)^2\le 0.
$$
Together with \eqref{Eqnq0} and \eqref{Eqnqbt}, this implies 
$$q(1)\ge0.$$
With \eqref{EqnQInterpolation} and \eqref{Eqnq}, this gives the desired inequality 
\begin{equation}
\label{EqnDesiredInequality}
\int_{\PB}\frac{|\nabla p|^2}{p}(p-\frac{1}{2d}|x|^2)=Q(1)=q(1)\ge0.
\end{equation}

The equality case in \eqref{EqnDesiredInequality} forces the equality in \eqref{Eqnq0}, this implies
$$
p=P_d=\frac{1}{2d}|x|^2.
$$
\end{proof}

\section{Concentration around cones with symmetry}
\label{SectionConesWithSymmetry}

In this section, we give the proof of our main result, Theorem \ref{ThmMain}. 

The first ingredient is the stability of the quantity in Lemma \ref{LemMagicLemma}:
\begin{lem}
\label{LemMagicLemmaForV}
Suppose $u\in\KGa$ with $\gamma\in[\frac12,1)\cup(1,\frac32]$ satisfies
\begin{equation}
\label{EqnVMinusP}
\|u-p\|_{\mathcal{L}^\infty(B_1)}<\frac{1}{4d}
\end{equation}
for some $p\in\P$.

Then 
$$
\int_{\partial B_1}\frac{|\nabla u|^2}{u}(p-\frac{1}{2d}|x|^2)\le  C_d\|u-p\|_{\mathcal{L}^\infty(B_1)}
$$
for a dimensional constant $C_d$.
\end{lem} 
Recall the spaces of cones, $\KGa$ and $\P$, from Definition \ref{DefCones} and Definition \ref{DefObPCones}.

\begin{proof}
\textit{Step 1: Consequence of Green's identity.}

With the $2$-homogeneity of  $u$, $p$ and $\frac{1}{2d}|x|^2$, and the harmonicity of $(p-\frac{1}{2d}|x|^2)$, we have
\begin{align*}
\int_{B_1}(p-\frac{1}{2d}|x|^2)\Delta(u-p)&=\int_{\partial B_1}(p-\frac{1}{2d}|x|^2)\nabla(u-p)\cdot\nu-(u-p)\nabla(p-\frac{1}{2d}|x|^2)\cdot\nu\\
&=2\int_{\partial B_1}(p-\frac{1}{2d}|x|^2)(u-p)-(u-p)(p-\frac{1}{2d}|x|^2)\\
&=0.
\end{align*}
With Proposition \ref{PropTransformed} and the definition of $\P$, the difference $(u-p)$ solves
\begin{equation}
\label{EqnLaplacianVMinusP}
\Delta(u-p)=-\chi_{\{u=0\}}+\frac{2-\beta}{2}\frac{|\nabla u|^2}{u} \text{ in }\R^d.
\end{equation}
Combining these two equations and using the homogeneity of the functions, we have
\begin{equation}
\label{EqnLastOfGreen}
\frac{2-\beta}{2}\int_{\partial B_1}\frac{|\nabla u|^2}{u}(p-\frac{1}{2d}|x|^2)=\int_{\partial B_1}(p-\frac{1}{2d}|x|^2)\chi_{\{u=0\}}.
\end{equation}

\vem

\textit{Step 2: The case when $\gamma>1$.}

In this case, Proposition \ref{PropConvexityTransformed} implies that either $u=\frac{2-\gamma}{2\gamma}[(x\cdot e)_+]^2$  or $\{u=0\}$ has zero measure. The former possibility is ruled out by \eqref{EqnVMinusP}. Combining this with \eqref{EqnLastOfGreen} and  $\beta\neq2$ (see \eqref{EqnBeta}), we get
$$
\int_{\partial B_1}\frac{|\nabla u|^2}{u}(p-\frac{1}{2d}|x|^2)=0,
$$
completing the proof in this case. 

\vem

\textit{Step 3: The case when $\gamma<1$.}

With \eqref{EqnVMinusP}, we have $p<\frac{1}{4d}$ on $\{u=0\}\cap B_1$. With the bound on $\Diri/u$ in Proposition \ref{PropTransformed}, we use \eqref{EqnLastOfGreen} to get
$$
\int_{\partial B_1}\chi_{\{u=0\}}\le C_d|2-\beta|.
$$
Together with \eqref{EqnLastOfGreen}, we have
\begin{align*}
\frac{2-\beta}{2}\int_{\partial B_1}\frac{|\nabla u|^2}{u}(p-\frac{1}{2d}|x|^2)&\le\int_{\partial B_1}p\chi_{\{u=0\}}\\
&\le\|u-p\|_{\mathcal{L}^\infty(B_1)}\int_{\partial B_1}\chi_{\{u=0\}}\\
&\le C_d |2-\beta|\cdot\|u-p\|_{\mathcal{L}^\infty(B_1)}.
\end{align*}

From \eqref{EqnBeta}, we see that $2-\beta>0$ in this case. Thus
$$
\int_{\partial B_1}\frac{|\nabla u|^2}{u}(p-\frac{1}{2d}|x|^2)\le C_d\|u-p\|_{\mathcal{L}^\infty(B_1)}.
$$
\end{proof} 

As a consequence, we have
\begin{cor}
\label{CorConcentration}
For $\gamma_n\neq 1$ with $\gamma_n\to 1$, suppose that a sequence $u_n\in\K(\gamma_n)$ satisfies
$$
u_n\to p \text{ locally uniformly in }\R^d \text{ for some }p\in\P.
$$

Then $p\in\SP$.
\end{cor} 
Recall the spaces of cones, $\KGa$, $\P$ and $\SP$, from Definition \ref{DefCones} and Definition \ref{DefObPCones}.
\begin{proof}
With Proposition \ref{PropTransformed}, we see that $u_n\to p$ locally  uniformly in $C^1(\R^d)$.

With $p\in\P$, we see that, up to a rotation, we have
$
p\ge\frac{1}{2d}x_1^2.
$
Given $\eta>0$, for large $n$, we have $u_n(x)>\frac{1}{4d}\eta^2$ when $|x_1|\ge\eta$. 

As a result,  we have
$$
\int_{\PB\cap\{|x_1|\ge\eta\}}\frac{|\nabla p|^2}{p}(p-\frac{1}{2d}|x|^2)=\lim\int_{\PB\cap\{|x_1|\ge\eta\}}\frac{|\nabla u_n|^2}{u_n}(p-\frac{1}{2d}|x|^2).
$$

Meanwhile, Proposition \ref{PropTransformed} implies that the integrals on $\PB\cap\{|x_1|\le\eta\}$ is controlled by $C_d\eta$ for some dimensional constant $C_d$.  Thus we have
$$
\int_{\PB}\frac{|\nabla p|^2}{p}(p-\frac{1}{2d}|x|^2)\le\limsup\int_{\PB}\frac{|\nabla u_n|^2}{u_n}(p-\frac{1}{2d}|x|^2)+C_d\eta\le C_d\eta.
$$
The last inequality follows from Lemma \ref{LemMagicLemmaForV}.

Sending $\eta\to0$, the conclusion follows from the equality case in Lemma \ref{LemMagicLemma}.
\end{proof} 

With this, we give the proof of our main result:
\begin{proof}[Proof of Theorem \ref{ThmMain}]
Suppose the theorem is false, then for some $\eps>0$, we  find a sequence  $\{v_n\}$, each being a minimizing cone to \eqref{EqnAP} with exponent $\gamma_n$ such that 
$$
\gamma_n\neq 1 \text{ for all }n, \hem\gamma_n\to1,
$$
but 
\begin{equation}
\label{Eqn}
\|v_n-p\|_{L^\infty(B_1)}\ge\eps
\end{equation}
for all $p\in\HS\cup\SP$ (see Definition \ref{DefObPCones}) and all $n\in\N.$

Suppose that $u_n$ is related to $v_n$ by \eqref{EqnTransformedSolution}. It is elementary to see from \eqref{EqnBeta} and \eqref{EqnTransformedSolution} that
$\lim_n\|u_n-v_n\|_{L^\infty(B_1)}=0.$
As a consequence of \eqref{Eqn}, for all large $n$,  we have 
\begin{equation}
\label{EqnOfContradiction}
\|u_n-p\|_{L^\infty(B_1)}\ge\eps/2
\end{equation}
for all $p\in\HS\cup\SP$.

\vem

With Proposition \ref{PropObPCones} and   Proposition \ref{PropCompactness}, we have, up to a subsequence, 
$$
u_n\to u_\infty \hem \text{ locally uniformly in }\R^d
$$
for some $u_\infty\in\HS$ or $u\in\P$ (see Definition \ref{DefObPCones}).  
 The first possibility is ruled out by \eqref{EqnOfContradiction}. Thus $u_\infty\in\P$. 

In this case, Corollary \ref{CorConcentration} implies that $u_\infty\in\SP$, contradicting \eqref{EqnOfContradiction}.
\end{proof}

\section*{Data availability statement}
There is no associated data for this work.


\end{document}